\documentclass[11pt]{article}

\usepackage{amsmath,amssymb,amsthm,bm,url}

\usepackage[pdftex]{graphicx}

\usepackage{ifthen}
\usepackage[english]{babel}
\usepackage[applemac]{inputenc}

\theoremstyle{plain}      \newtheorem{lem}{Lemma}[section]
                          \newtheorem{thm}[lem]{Theorem}
                          \newtheorem{prop}[lem]{Proposition}
                          \newtheorem{cor}[lem]{Corollary}
                           \newtheorem{rem}[lem]{Remark}
                           \newtheorem{exa}[lem]{Example}

\theoremstyle{definition} \newtheorem{defin}{Definition}[section]

\date{}

\numberwithin{equation}{section}

\usepackage{varioref}
\labelformat{section}{Section~#1}
\labelformat{figure}{Figure~#1}
\labelformat{table}{Table~#1}
\labelformat{defin}{Definition~#1}

\usepackage[numbers,sort]{natbib}
\newcommand{\N}{\mathbb N}
\newcommand{\R}{\mathbb R}
\newcommand{\C}{\mathbb C}
\newcommand{\E}{\mathfrak E}
\renewcommand{\c}{\mathfrak c}

\newcommand\Spec{\textup{Spec}}
\newcommand\vol{\textup{vol}}
\newcommand\dist{\textup{dist}}

\renewcommand\div{\textup{div}}
\newcommand\supp{\textup{supp}}
\newcommand{\indic}{\operatorname{1\negthinspace l}}

\def\<{\langle}
\def\>{\rangle}

\begin{document}

\title{Gibbs/Metropolis algorithms on a convex polytope}

\author{\textsc{Persi Diaconis}\footnote{Supported in part by NSF grant 0804324}\and%
          \textsc{Gilles Lebeau}\footnote{Supported in part by ANR-06-BLAN-0250-03}\ %
\footnote{Corresponding author: Parc Valrose 06108 Nice Cedex 02, France; \texttt{lebeau@unice.fr}}\and%
\textsc{Laurent Michel}\footnote{Supported in part by ANR-06-BLAN-0250-03 and ANR-09-JCJC-0099-01}\and\\\textit{Departments of
      Mathematics and Statistics,}\\\textit{Stanford University}\\ and\\\textit{D\'epartement de
              Math\'ematiques,}\\\textit{Universit\'e de Nice Sophia-Antipolis}}

\maketitle

\begin{abstract}
  This paper gives sharp rates of convergence for natural versions of
  the Metropolis algorithm for sampling from the uniform distribution
  on a convex polytope. The singular proposal distribution, based on a
  walk moving locally in one of a fixed, finite set of directions,
  needs some new tools. We get useful bounds on the spectrum and
  eigenfunctions using Nash and Weyl-type inequalities. The top
  eigenvalues of the Markov chain are closely related to the Neuman
  eigenvalues of the polytope for a novel Laplacian.
\end{abstract}

\section{Introduction}\label{sec1}
\subsection{Overview}\label{sec11}

The Metropolis algorithm and the Gibbs sampler (also known as Glauber
dynamics) are often used together as one of the basic tools of
scientific computation. We treat the following example: let $\Omega$
be a polyhedral convex set in $d$ dimensions. To sample from the
uniform distribution on $\Omega$, from a point $x$ in $\Omega$, pick a
direction $e$ from a fixed finite collection. Set $y=x+ue$ where $u$
is chosen uniformly in $[-h,h]$. If $y\in\Omega$, move to $y$. Else,
stay at $x$. Under a mild generality condition on the set of
directions in relation to $\Omega$, this Markov chain converges to the
uniform distribution on $\Omega$. Our main result gives a sharp
determination of the exponential rate of convergence of this
algorithm. It is $ce^{-ng(h)}$ with $g(h)$ asymptotic to $h^2\nu$ for
$\nu$ the first non zero  eigenvalue of a novel Laplacian defined on $\Omega$ with Neumann condition on the boundary.

Sampling from a convex set is a practical problem. For example,
choosing a uniformly distributed $100\times100$ doubly stochastic
matrix \cite{pd178} or a uniformly distributed $100\times100$
tri-diagonal doubly stochastic matrix \cite{pd179}. It is also a basic
problem of study in theoretical computer science
\cite{lovasz93,lovasz03}. Many algorithms have been proposed and
studied. A readable textbook description of the Gibbs sampler is in
Liu \cite{liu}. See \cite{pd56} for a review of rigorous results for
the Metropolis algorithm in finite spaces. The popular \textit{hit and
  run algorithm} \cite{smith,pd168,lovasz07} was introduced for this
purpose. Hit and run makes long moves and will probably be preferred
in practice to the local algorithms studied here.

Spectral techniques for analysis of the Metropolis algorithm on
continuous spaces are developed in \cite{pd165,lebeau,pd167}. The
proposal distributions there are ``ball walks'' choosing from the
uniform distribution on the interior of a ball. The discrete set of
directions studied here is widely used in practice and necessitates
new ideas. Present problems can also be studied by Harris recurrence
techniques \cite{meyn,rosen} and by the path techniques of Yuen
\cite{yuen}. These give useful results but do not get the sharp rates
on the exponents derived here.

The remainder of this section gives a careful description of the
Markov chain and the geometric connection between the underlying
directions and the convex set $\Omega$ required for ergodicity.
\ref{sec2} gives bounds on the spectrum and eigenvectors using Nash
inequalities and Weyl-type inequalities. \ref{sec3} uses this spectral
information to get rates of convergence. \ref{sec4} proves that our
operator (suitably rescaled) converges, in the strong resolvent sense,
to a novel Laplace operator on $\Omega$ with Neumann boundary
conditions. A similar convergence of the ball walk Metropolis operator
to the usual Neumann Laplacian is a key ingredient of
\cite{pd167,lebeau}. The final section shows how to modify the
argument to handle a continuous choice of direction.

\subsection{Basic definitions}\label{sec12}

Let $\Omega$ be an open convex polytope in $\R^d$, $d\geq 2$. Thus there
exists linear forms $\ell_j:\R^d\to \R,\ j=1,\dots,m$ and real numbers
$b_j$ such that
\begin{equation}
\Omega=\left\{x\in\R^d,\;\forall j=1,\dots,m,\;\ell_j(x)>b_j\right\}
\label{11}
\end{equation}
Assume also that $\Omega$ is bounded and non empty.

Consider $\E=\{e_1,\dots,e_p\}$ a family of vectors in $\R^d$.  For
any $j\in\{1,\dots,p\}$ we introduce the operator acting on continuous
functions $M_{j,h}f(x)=m_{j,h}(x)f(x)+K_{j,h}f(x)$, where
\begin{equation}
K_{j,h}(f)(x)=\dfrac12\int_{t\in[-1,1]}1_\Omega(x+hte_j)f(x+hte_j)\,dt
\label{12}
\end{equation}
and $m_{j,h}(x)=1-K_{j,h}(1)(x)$.

The local Metropolis operator associated to the family $\E$ is
\begin{equation}
M_h(f)(x)=\dfrac1{p}\sum_{j=1}^p\left(K_{j,h}f(x)+m_{j,h}(x)f(x)\right).
\label{13}
\end{equation}
In the sequel, denote $m_h=\tfrac1{p}\sum_{j=1}^pm_{j,h}$ and
$K_h=\tfrac1{p}\sum_{j=1}^pK_{j,h}$. Let $M_h(x,dy)$ be the Markov
kernel associated to this operator. This defines a bounded
self-adjoint operator on $L^2(\Omega)$. Moreover, since $M_h(1)=1$,
$\parallel M_h\parallel_{L^2\to L^2}=1$. Thus the probability measure
$\frac{dx}{\vol(\Omega)}$ on $\Omega$ is stationary. For $n\geq1$,
denote by $M_{h}^n(x,dy)$ the kernel of the iterated operator
$(M_{h})^n$. For any $x\in \Omega$, $M_{h}^n(x,dy)$ is a probability
measure on $\Omega$, and our main goal is to get some estimates on the
rate of convergence, when $n\to+\infty$, of the probability
$M_{h}^n(x,dy)$ toward the stationary probability
$\frac{dy}{\vol(\Omega)}$.

A good example to keep in mind is the case where $\Omega=A_N$ is the
set of $N\times N$ doubly stochastic matrices. In other words,
\begin{equation}
A_N=\left\{(a_{i,j})_{1\leq i,j\leq N},\;\forall i,j,\;a_{i,j}>0,\;\sum_{k}a_{ik}=\sum_{k}a_{kj}=1\right\}.
\label{14}
\end{equation}
The set $A_N$ can be viewed as convex open polytope in 
$A_N^0=\{(a_{i,j})_{1\leq i,j\leq
  N},\;\sum_{k}a_{ik}=\sum_{k}a_{kj}=1\}$. A good way to sample from
$A_N$ is to use the Metropolis strategy in the following manner.
Starting from a matrix $A\in A_N$ choose two distinct rows $R_{i_1},R_{i_2}$
and two distinct columns $C_{j_1}, C_{j_2}$ at random. Denote
$\vec{i}=(i_1,i_2,j_1,j_2)$ and $F=F(\vec{i})$ the matrix such that
$F_{i,j}=\delta_{i_1 j_1}-\delta_{i_1,j_2}-\delta_{i_2
  j_1}+\delta_{i_2 j_2}$. For $h>0$ given, build the family of
matrices $(\tilde A(t)=A+tF(\vec{i}))_{t\in[-h,h]}$. For any $t\in \R$
the matrix $\tilde A(t)$ belongs to the set $A_N^0$. Taking
$t\in[-h,h]$ at random and keeping the move $A\to \tilde A(t)$ only if
it results in an element of $A_N$, we are exactly in the above
situation with $\E=\{F(\vec{i})\}$. This
algorithm is used in \cite{pd178} to study things like the
distribution of typical entries or the eigenvalues of random doubly
stochastic matrices.

Let us go back to the general problem. From the definition of
$\Omega$, a point $x\in\R^d$ belongs to $\partial\Omega$ iff there
exists a partition $I\cup J=\{1,\dots,m\}$ such that $I\neq\emptyset$ and 
\begin{equation}
\forall i\in I,\;\ell_i(x)=b_i\quad\text{and}\quad\forall j\in J,\;\ell_j(x)>b_j.
\label{15}
\end{equation}
Define the following function $\c:\R^d\to\N\cup\{+\infty\}$ by
\begin{alignat}{3}
\c(x)&=0\quad&\text{if }x&\in\Omega\notag\\
&=+\infty\quad&\text{if }x&\in\R^d\setminus\overline\Omega\label{16}\\
&=\text{card}(I)\quad&\text{if } x&\in\partial\Omega.\notag
\end{alignat}
To proceed, the following geometric condition is needed; it shows how
the generating set $\E$ must be related to the convex set $\Omega$.
Proposition \ref{prop15} shows the condition is equivalent to $M_h$ having a
spectral gap.
\begin{defin}
  The family $\E$ is weakly incoming to the set $\Omega$ if for any
  point $x_0\in\partial\Omega$ there exists $\epsilon>0,\
  \theta\in\{\pm1\}$ and $e\in\E$ such that, for $\c$ defined in
  \eqref{16},
\begin{equation}
\c(x_0+\theta te)<\c(x_0)\qquad\forall t\in]0;\epsilon].
\label{17}
\end{equation}
\label{def11}
\end{defin}

The following observation is simple and fundamental. Suppose that $\E$
is weakly incoming, then span$(\E)=\R^d$. Indeed, otherwise there is a
hyperplane $H=(\R\nu)^\bot$ of $\R^d$ such that span$(\E)\subset H$.
Since $\overline\Omega$ is compact, the function
$x\in\overline\Omega\mapsto\<x,\nu\>$ would have a global minimum in
some $x_0\in\partial\Omega$. Since $\Omega$ is open, $\Omega\subset
x_0+H^+$, where $H^+=\{y\in\R^d,\<y,\nu\>>0\}$. As $\E$ is weakly
incoming, there is $u\in\text{span}(\E)$ such that $\c(x_0+u)=0$. In
other words, $x_0+u\in\Omega\cap(x_0+H)$. This contradicts
$\Omega\subset x_0+ H^+$.

\begin{exa}
Consider $\Omega$ the convex hull of on equilateral triangle $(ABC)$ in $\R^2$ and $\E=\{e_1,e_2\}$ like on  \ref{fig1}. For $\alpha\in]0,\pi/3]$, $\E$ is weakly incoming to $\Omega$ whereas for $\alpha\in]\pi/3,\pi[$, condition \eqref{17} is satisfied in every point $x_0$ of the boundary excepted in point $A$.
\end{exa}

\begin{figure}[htbp]
\centering
\includegraphics[scale=0.5, trim=0in 0in 0in 0in, clip=true]{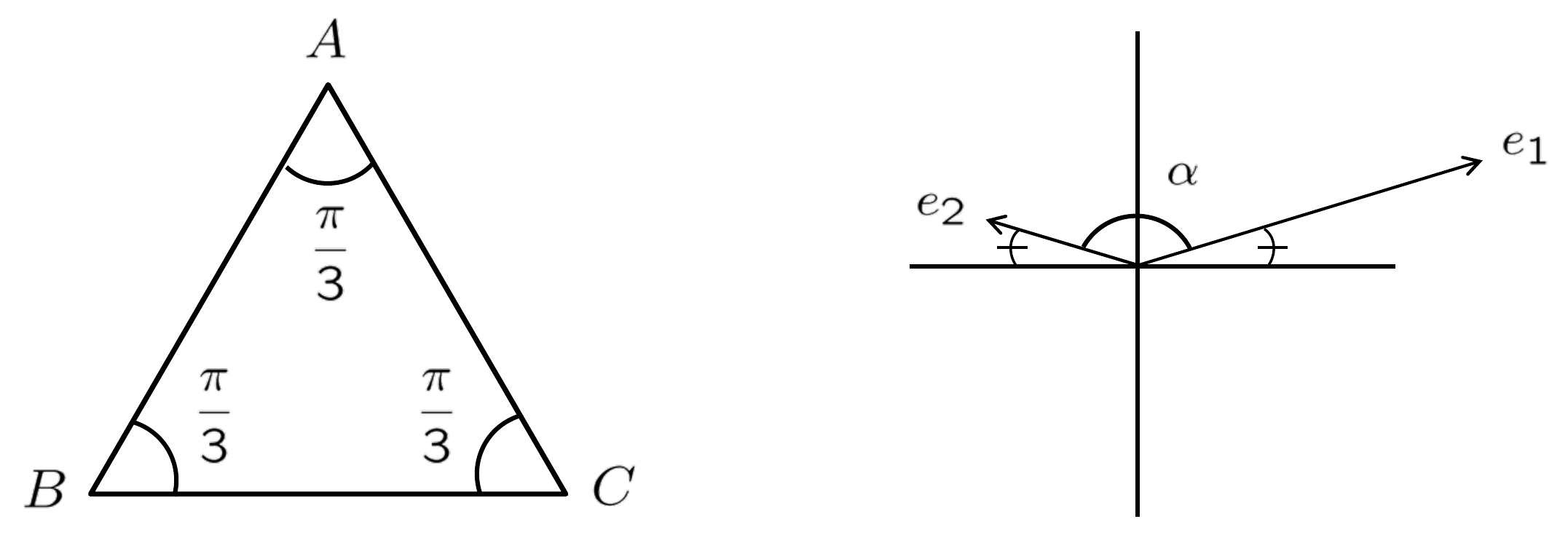}
\caption{Weakly incoming condition in the case of an equilateral triangle.}
\label{fig1}
\end{figure}

\begin{rem}
In the above case of doubly stochastic matrices, the set $\E=\{F(\vec{i})\}$ is weakly incoming. Indeed, if $A$ is in the boundary of $A_N$, there exists $i_1$ and $j_1$ such that $A_{i_1j_1}=0$. Since $A$ is doubly stochastic, there exists $i_2,j_2$ such that $A_{i_1j_2}>0$ and $A_{i_2j_1}>0$. Let $\epsilon=\min(A_{i_1j_2},A_{i_2j_1})/2$, then for all $t\in]0,\epsilon]$, $\c(A+tF(i_1,i_2,j_1,j_2))<\c(A)$.
\end{rem}

Denote $H_k=\text{ker}(\ell_k)$ and let $\nu_k$ be
the unit vector such that $\ell_k(\nu_k)>0$ and $H_k^+=\{y\in\R^d,\,\ell_k(y)>b_k\}$.

\begin{defin}
  Let $u\in\R^d\setminus\{0\}$. The vector $u$ is incoming to $H_k$ if
  $\<u,\nu_k\>\geq 0$. Further, $u$ is strictly incoming to $H_k$ if
  $\<u,\nu_k\>> 0$; $u$ is strictly outgoing to $H_k$ if
  $\<u,\nu_k\><0$; $u$ is parallel to $H_k$ if
  $u\in H_k$.
\label{def12}
\end{defin}
\begin{lem}
  Suppose that $\E$ is weakly incoming to $\Omega$ and let
  $x_0\in\overline\Omega$ and $k=\c(x_0)$. There exists $r>0$ and
  $I\subset\{1,\dots,m\}$ such that $\sharp I=k$ and  for
  $B(x_0,r)$ the open ball of radius $r$ about $x_0$,
\begin{equation}
\Omega\cap B(x_0,r)=B(x_0,r)\cap\left(\cap_{i\in I} H_{i}^+\right).
\label{18}
\end{equation}
Further, there exists $\beta_1,\dots,\beta_k,\in\{1,\dots,p\}$, a
family $(\theta_n)_{n=1,\ldots, k}$ of numbers in $\{\pm1\}$ and a bijection $\{1,\ldots, k\}\ni n\mapsto i_n \in I$ such that for
all $n\in\{1,\ldots, k\}$,
\begin{equation}\begin{gathered}
\theta_ne_{\beta_n}\text{ is strictly incoming to }H_{i_n};\\
\theta_ne_{\beta_n}\text{ is  incoming to }H_{i_m},\quad\forall m>n.
\end{gathered}\label{19}
\end{equation}
\label{lem13}
\end{lem}
\begin{proof}
Proceed by induction on $k=\c(x_0)$. When $k=0$, there is nothing to prove.

Suppose that the property holds true at rank $k'\leq k-1$ and let
$x_0\in\partial\Omega$ be such that $\c(x_0)=k$. By definition of
$\Omega$, there exists $r>0$ and
$I\subset\{1,\dots,m\}$ with $\sharp I=k$ such that
\begin{equation}
\Omega\cap B(x_0,r)=B(x_0,r)\cap(\cap_{i\in I} H_{i}^+)
\label{110}
\end{equation}
Since $\E$ is weakly incoming to $\Omega$, there exists
$q\in\{1,\dots,k\},\ \theta_1=\pm1$, $\beta_1\in\{1,\dots,p\}$ and $i_1,\ldots,i_q\in I$ such
that $\theta_1e_{\beta_1}$ is strictly incoming to $H_{i_n}$ for
$n=1,\dots,q$ and $\theta_1e_{\beta_1}$ is parallel to $H_{i}$
for $i\in I':=I\setminus\{i_1,\ldots,i_q\}$. By definition of $\Omega$ there exists $x'_0$
close to $x_0$ and $r'>0$ such that
$B(x'_0,r')\cap\Omega=B(x_0',r')\cap(\cap_{i\in I'} H_{i}^+)$. From the induction
hypothesis, there exists $\beta_{q+1},\dots,\beta_{k}\in\{1,\ldots,p\}$,
$\theta_{q+1},\ldots,\theta_k=\pm 1$ and a bijection $\{q+1,\ldots,k\}\ni n\mapsto i_n \in I'$ such that for all $n\geq q+1$,
\begin{equation}\begin{gathered}
\theta_ne_{\beta_n}\text{ is strictly incoming to }H_{i_n};\\
\theta_ne_{\beta_n}\text{ is incoming to }H_{i_m}\qquad\forall m>n.
\end{gathered}\label{111}
\end{equation}
and the proof is complete.
\end{proof}
\begin{cor} 
  Suppose that $\E$ is weakly incoming to $\Omega$. There exist $r>0$ and $ \epsilon\in]0,1],$ such that for all $x_0\in\overline\Omega$, there exists $q\in
  \{1,\dots,p\}$ and $\theta_q=\pm1$ such that
\begin{equation}
x+t\theta_qe_q\in\Omega\qquad\forall x\in B(x_0,r)\cap\Omega,\ \forall t\in[0,\epsilon].
\label{112}
\end{equation}
\label{cor14}
\end{cor}
\begin{proof}
The fact that $r,\epsilon>0$ can be chosen uniformly with respect to $x_0$ follows easily from compactness of $\overline\Omega$.
  The statement is trivial when $x_0\in \Omega$. Suppose that
  $x_0\in\partial\Omega$.  From Lemma \ref{lem13}, there exists $r>0$
  and $I=\{i_1,\dots,i_k\}\subset\{1,\dots,p\}$ such that
\begin{equation}
\Omega\cap B(x_0,2r)=B(x_0,2r)\cap\left(\cap_{i\in I}H_{i}^+\right)
\label{113}
\end{equation}
and $\theta_1=\pm1,\ \beta_1\in\{1,\dots,p\}$ such that
\begin{equation}\begin{gathered}
\theta_1e_{\beta_1}\text{ is strictly incoming to }H_{i_1};\\
\theta_1e_{\beta_1}\text{ is incoming to }H_{i_q},\;\forall q>1.
\end{gathered}\label{114}
\end{equation}
Let $x\in B(x_0,r)\cap\Omega$ and $\epsilon\in]0,r[$. Then
\begin{equation}
\<\nu_{i},\theta_1{\beta_1}\>\geq0,\,\forall i\in I\Longrightarrow x+t\theta_1e_{\beta_1}\in\Omega.
\label{115}
\end{equation}
Thanks to \eqref{114}, the left hand side of the above property is satisfied and the proof is
complete.
\end{proof}
\begin{prop} 
  The family $\E$ is weakly incoming to $\Omega$ iff $1$ is not in the
  essential spectrum of $M_h$.
\label{prop15}
\end{prop}
\begin{proof}
  If $\E$ is weakly incoming to $\Omega$, $1$ is not in the essential
  spectrum of $M_h$ thanks to Proposition \ref{prop22} of this paper
  and Theorem 1.1 in \cite{pd167}.

  Suppose now that $\E$ is not weakly incoming to $\Omega$. This means
  that there exists $x_0\in\partial\Omega$ such that \eqref{17} does
  not hold. Let $k=\c(x_0)$. There exists a neighborhood $V$
  of $x_0$ and $I\subset\{1,\dots,m\}$ with $\sharp I=k$
  such that $V\cap\Omega=V\cap(\cap_{i\in I}H_{i})$. Then, for
  any $\theta=\pm1$ and any $j\in\{1,\dots,p\}$, the following holds
  true:
\begin{center}
If $\theta e_j$ is strictly incoming to one of the $(H_i)_{i\in I}$, \\
then $\theta e_j$ is strictly outgoing to one of the $(H_i)_{i\in I}$.
\end{center}
Otherwise, there is $j\in\{1,\dots,p\}$ and $\theta=\pm1$ such that
$\theta e_j$ is strictly incoming to one of the $(H_i)_{i\in I}$ and
incoming to the other. Then for $t>0$ small enough, $\c(x_0+\theta
te_j)<\c(x_0)$.

Hence, assume that there exists $r\geq1$ such that 
\begin{itemize}
\item for any $j\in\{1,\dots,r\},\ e_j$ and $-e_j$ are strictly
  outgoing to some of the $(H_i)_{i\in I}$;
\item for any $j\in\{r+1,\dots,p\},\ e_j$ is parallel to the
  $(H_i)_{i\in I}$.
\end{itemize}
Recall that $\nu_i$ denotes the unit incoming orthogonal vector to
$H_i$. Let $W=\text{span}(\nu_i,i\in I)$ and near $x_0$ use the
variable $x=x_0+(x',x'')$ with $x'\in W$ and $x''\in W^\bot$. Let
$\chi(x')=\indic_{\tfrac12<| x'|<1}$ and for $\lambda,h>0$ denote
$f_{\lambda h}(x)=(\lambda h)^{-dim(W)/2}\chi(\tfrac{x'}{\lambda h})$.
Since any $v\in W^\bot $ is parallel to the $(H_i)_{i\in I}$, there
exists $\lambda_0,c_0>0$ such that for all $h\in]0,1]$ and
$\lambda\in]0,\lambda_0],\ \Vert f_{\lambda
  h}\Vert_{L^2(\Omega)}\geq c_0$.

For any $j\in\{r+1,\dots,p\},\ e_j$ is parallel to the $(H_i)_{i\in
  I}$. Hence, the function $t\mapsto f_{\lambda h}(x+hte_j)$ is
constant and $(M_{j,h}-1)f_{\lambda h}(x)=0$.

On the other hand, for any $j\in\{1,\dots,r\}$ there exists
$i_j,i'_j\in I$ such that $e_j$ is strictly outgoing to
$H_{i_j}$ and $-e_j$ is strictly outgoing to $H_{i'_j}$.
Consequently, there exists $\gamma_j,\delta_j>0$ such that for $t>0$,
\begin{equation}\begin{aligned}
x\in\Omega\text{ and }x-te_j\in\Omega\Longrightarrow\dist(x-te_j,H_{i'_j})&\leq\dist(x,H_{i'_j})-\gamma_j t\\
x\in\Omega\text{ and }x+te_j\in\Omega\Longrightarrow\dist(x+te_j,H_{i_j})&\leq\dist(x,H_{i_j})-\delta_j t.
\end{aligned}\label{116}
\end{equation}
Let us compute the potential $m_{j,h}$ on the support of $f_{\lambda
  h}$. For $x\in\supp(f_{\lambda h})$, $|x_j|\leq\lambda h$ for
all $j=1,\dots,r$. In particular dist$(x,H_{i_j})\leq\lambda h$
and dist$(x,H_{i'_j})\leq\lambda h$ and thanks to \eqref{116},
\begin{equation}\begin{split}
1-m_{j,h}(x)&=\int_0^11_\Omega(x+hte_j)+1_\Omega(x-hte_j)\,dt\\
&\leq\int_{0\leq t\leq\dist(x,H_{i_j})/(\delta_j h)}\,dt+ 
   \int_{0\leq t\leq\dist(x,H_{i'_j})/(\gamma_j h)}\,dt\leq\lambda\left(\dfrac1{\gamma_j}+\dfrac1{\delta_j}\right).
\end{split}\label{117}
\end{equation}
Finally, 
\begin{equation}\begin{split}
\left\<(1-M_h)f_{\lambda h},f_{\lambda h}\right\>_{L^2(\Omega)}&=\dfrac1{p}\sum_{j=1}^r\left\<(1-M_{j,h})f_{\lambda h},f_{\lambda h}\right\>_{L^2(\Omega)}\\
&\leq\dfrac1{p}\sum_{j=1}^r\int_\Omega\left(1-m_{j,h}(x)\right)|f_{\lambda h}(x)|^2\,dx\leq C\lambda\|f_{\lambda h}\|_{L^2(\Omega)}^2.
\end{split}\label{118}
\end{equation}
Here we used the fact that for any non-negative fonction $f$, one has $\<K_{j,h}f,f\>\geq 0$. Finally, we conclude by taking $\lambda=2^{-n}\to0$ as $n\to\infty$. Indeed, the
functions $f_{2^{-n}h}$ are mutually orthogonal. Their norm is bounded
uniformly from below and they satisfy
$0\leq\<(1-M_h)f_{2^{-n}h},f_{2^{-n} h}\>\leq C2^{-n}$.
\end{proof}

\section{Spectral Analysis of the Metropolis Operator}\label{sec2}

This section is devoted to the analysis of the spectral theory of the
Metropolis operator. For this purpose, we introduce a Laplace operator
associated to the family $\E$ to be used as a model. For any $e\in
\R^d\setminus\{0\}$ and any smooth function $u$, define
$\partial_eu(x)=\tfrac{d}{dt}(u(x+te))_{| t=0}$. Then, consider the
operator $\Delta_\E$, defined by
\begin{equation}\begin{aligned}
\Delta_\E u&=\dfrac1{6p}\sum_{j=1}^p\partial_{e_j}^2u\\
D(\Delta_\E)&=\left\{u\in H^1(\Omega),\,\Delta_\E u\in L^2,\partial_{n,\E}u_{|\partial\Omega}=0\right\}
\end{aligned}\label{21}
\end{equation}
with
$\partial_{n,\E}u(x)=\tfrac1{p}\sum_{j=1}^p\<n(x),e_j\>\partial_{e_j}u(x)$,
$n(x)$ denoting the outgoing normal vector to the boundary at point
$x$. If the domain $\Omega$ has smooth boundary, the normal derivative
is well defined.  In the case where it is Lipschitz, it can be defined
by duality in the following way.

Define first the gradient and divergence associated to the family
$\E$, by $\div_\E u=\tfrac1{p}\sum_{j=1}^p\partial_{e_j}u_j$ for any
$u=(u_1,\dots,u_p)$ and $\nabla_\E
u=(\partial_{e_1}u,\dots,\partial_{e_p}u)$. Then, define a trace
operator $\gamma_\E$ by
\begin{equation}
\gamma_\E:\left\{u\in\left(L^2(\Omega)\right)^p,\div_\E(u)\in L^2(\Omega)\right\}\to H^{-1/2}(\partial\Omega)
\label{22}
\end{equation}
and and for $v\in H^1(\Omega)$, 
\begin{equation}
\int_{\Omega}\div_\E(u)(x)v(x)\,dx=-\frac1{p}\int_{\Omega}\left\<u(x),\nabla_\E v(x)\right\>_{\C^p}\,dx+
\int_{\partial\Omega}\gamma_\E(u)v|_{\partial\Omega}\,d\sigma(x).
\label{23}
\end{equation}
In particular, for $u\in H^1(\Omega)$ satisfying $\Delta_\E
u=\tfrac16\div_\E\nabla_\E u\in L^2(\Omega)$ define
$\partial_{n,\E}u|_{\partial\Omega}=\gamma_\E(\nabla_\E u)\in
H^{-1/2}(\partial\Omega)$ and the set $D(-\Delta_\E)$ is well defined.
The Dirichlet form associated with $-\Delta_\E$ is
\begin{equation}
\mathcal{E}_\E(u)=\dfrac1{6p}\sum_{j=1}^p\int_\Omega|\partial_{e_j}u(x)|^2\,dx.
\label{24}
\end{equation}

Let $\E_0$ be the canonical basis in $\R^d$. Then,
$\Delta_{\E_0}=\tfrac1{6d}\Delta$ where $\Delta$ is the usual Laplace
operator and $\mathcal{E}_{\E_0}(f)=\tfrac1{6d}\int_\Omega|\nabla
f|^2dx$ is the usual Dirichlet form. Since $\E=\{e_1,\dots,e_p\}$
spans $\R^d$, a simple calculation shows that there exists a constants
$C>0$ such that
\begin{equation}
C^{-1}\mathcal{E}_{\E_0}(f)\leq\mathcal{E}_{\E}(f)\leq C\mathcal{E}_{\E_0}(f).
\label{25}
\end{equation}

Then, it is standard to show that $-\Delta_\E$ is the self-adjoint
realization of the Dirichlet form $\mathcal{E}_\E$. A standard
argument using Sobolev embedding shows that $-\Delta_\E$ has compact
resolvant. Denote its spectrum by $\nu_0=0<\nu_1<\nu_2<\dots$ and
by $m_j$ the associated multiplicities. Observe that
$m_0=1$. \ref{sec4} shows that $h^{-2}(1-M_h)$ converges to
$-\Delta_\E$ in the strong resolvent sense so that eigenvalues and
eigenvectors converge; see \cite{reedIV}.

The main theorem of this section follows.
\begin{thm}
Suppose that $\E$ is weakly incoming to $\Omega$, then the following hold true.
\begin{enumerate}
\item[i)] There exists $h_0>0$, $\delta_0\in]0,\tfrac12[$ and a
  positive constant $C$ such that for any $h\in]0,h_0]$, the spectrum
  of $M_h$ is a subset of $[-1+\delta_0,1]$, $1$ is a simple
  eigenvalue and $\Spec(M_h)\cap[1-\delta_0,1]$ is discrete.

\item[ii)] For any $h\in]0,h_0]$ and
  $0\leq\lambda\leq\delta_0h^{-2}$, the number of eigenvalues of
  $M_{h}$ in $[1-h^2\lambda,1]$ (with multiplicity) is bounded by
  $C(1+\lambda)^{d/2}$.
  
\item[iii)] For any $R>0$ and $\varepsilon>0$ such that
  $\nu_{j+1}-\nu_j>2\varepsilon$ for $\nu_{j+2}<R$, there exists
  $h_1>0$ such that one has for all $h\in]0,h_1]$,
\begin{equation}
\Spec\left(\frac{1-M_{h}}{h^2}\right)\cap]0,R]\subset\cup_{j\geq1}[\nu_j-\varepsilon,\nu_j+\varepsilon],
\label{26}
\end{equation}
and the number of eigenvalues of $\tfrac{1-M_{h}}{h^2}$ in the interval
$[\nu_j-\varepsilon,\nu_j+\varepsilon]$ is equal to $m_j$.
\end{enumerate}
\label{thm21}
\end{thm}

A consequence of this theorem is that $M_h$ has a spectral gap
$g(h)=1-\sup(\Spec(M_h)\setminus\{1\})>0$ and that $\lim_{h\to
  0^+}h^{-2}g(h)=\nu_1$. This will be used in the proof of total
variation estimates.

The strategy used to prove the first part of Theorem \ref{thm21} is
very close to the one given in \cite{pd167}. First, show that some
iterate of the Markov kernel ``controls'' the random walk on a ball.
Next, this ball walk on the polytope is compared to the same walk on a
large torus containing $\Omega$. Finally the information on the torus
is transferred back to the original problem.

The proof of the last part of Theorem \ref{thm21} is slightly
different from the proof in \cite{pd167}. Indeed, the starting point
of the analysis in \cite{pd167} is that for regular function $\varphi$ with
normal derivatives vanishing on the boundary, $h^{-2}(1-T_h)\varphi$
is close to $-\Delta\varphi$ up to the boundary, where $T_h$ is the
Metropolis operator associated to the kernel
$\vol(B(0,1))^{-1}h^{-d}1_{| x-y|<h}$. Here, this property fails to be
true. Suppose for instance that $\Omega\subset \R^2$ and that its
boundary is given near $(0,0)$ by $x_1\geq 0$. Suppose that
$e_1=(a,b)$ and $e_2=(b,-a)$ for some $a,b>0$. Then
\begin{equation}\begin{split}
h^{-2}(1-M_{1,h})f(x)&=\dfrac1{2h^3}\int_{| t|<h,x+te_1\in\Omega}\left(f(x)-f(x+te_1)\right)\,dt\\
&=-\dfrac1{2h^3}\partial_{e_1}f(x)\int_{| t|<h,x+te_1\in\Omega}t\,dt+O(1)\\
&=\dfrac1{4h^3}\partial_{e_1}f(x)1_{]0,ah]}(x_1)\left(h^2-\frac{x_1^2}{a^2}\right)+O(1).
\end{split}\label{27}
\end{equation}
A similar expression holds for $M_{2,h}$ and summing these equalities
gives
\begin{multline}
h^{-2}(1-M_h)f(x)=\dfrac1{4h^3}\left(\partial_{e_1}f(0,x_2)1_{]0,ah]}(x_1)\left(h^2-\dfrac{x_1^2}{a^2}\right)\right.\\
+\left.\partial_{e_2}f(0,x_2)1_{]0,bh]}(x_1)\left(h^2-\dfrac{x_1^2}{b^2}\right)\right)+O(1)
\label{28}
\end{multline}
If $a=b,\ \partial_{e_1}f+\partial_{e_2}f$ is proportional to the
normal derivative of $f$ and hence, the above quantity is bounded.

Suppose now that $a<b$. Then the above quantity is bounded on
$x_1\in[ah,bh]$ provided $\partial_{e_2}f(0,x_2)=0$. Then the same
argument on $[0,ah]$ shows that $\partial_{e_1}f(0,x_2)=0$ also.

In order to avoid these difficulties, we work directly on the
quadratic form and show that the Dirichlet form associated to the
Metropolis operator converges to the Dirichlet form of the Laplace
operator with Neuman boundary conditions. The end of this section is
devoted to the proof of Theorem \ref{thm21}.
\begin{prop}
  There exists $N\in\N$ and constants $c_1,c_2>0$ such that for all
  $h\in]0,1]$
\begin{equation}
M_h^N(x,dy)=\mu_h(x,dy)+c_1h^{-d}1_{| x-y|<c_2h}\,dy
\label{29}
\end{equation}
where for all $x\in\Omega,\ \mu_h(x,dy)$ is a positive Borel measure.
\label{prop22}
\end{prop}
\begin{proof}
  The proof follows the lines of \cite{pd167}.  Denote
  $K_{h}=\tfrac1{p}\sum_{j=1}^pK_{j,h}$. 
  Since for any $h_2>h_1>0$ and any non-negative function $f$, $h_2K_{h_2}f\geq
h_1K_{h_1}f$,
  it is sufficient to prove the
  following: there exists $h_0>0$, $c_1,c_2>0$ and $N\in\mathbb{N}^*$ such that
  for all $h\in]0,h_0]$, one has, for all non-negative continuous
  functions $f$,
\begin{equation}
K_{h}^{N}(f)(x)\geq c_1h^{-d}\int_{ y\in\Omega, | x-y|\leq c_2h}f(y)\,dy.
\label{210}
\end{equation}
 First note that it is sufficient to
prove the weaker version: for all $x^0\in \overline\Omega$, there
exist $N(x^0),\alpha=\alpha(x^0)>0,\ c_1=c_1(x_0)>0,\ c_2=c_2(x_0)>0,\
h_0=h_0(x_0)>0$ such that for all $h\in]0,h_0]$, all $x\in\Omega$ and
all non-negative functions $f$
\begin{equation}
|x-x^0|\leq2\alpha\Longrightarrow K_{h}^{N(x_0)}(f)(x)\geq c_1h^{-d}\int_{ y\in\Omega,|x-y|\leq c_2h}f(y)\,dy.
\label{211}
\end{equation}
Let us verify that \eqref{211} implies \eqref{210}. Decreasing
$\alpha(x_0)$ if necessary, it may be assumed that
$2\alpha(x_0)<r(x_0)$, where $r(x_0)$ is given by Lemma \ref{lem13}.
Since $\overline\Omega$ is compact, there exists a finite set $F$ such
that $\overline\Omega\subset\cup_{x_0\in F}\{|x-x_0|<\alpha(x_0)\}$.
Let $N=\sup\{N(x_0),\,x_0\in F\}$, $c'_i=\min_{x_0\in F}c_i(x_0)$ and
$h'_0=\min_{x_0\in F}h_0(x_0)$. One has to check that for any $x_0\in
F$ and any $x$ with $|x-x_0|\leq\alpha(x_0)$, the right inequality in
\eqref{211} holds true with $N=N(x_0)+n$ in place of $N(x_0)$ for some
constants $c_1,c_2,h_0$. Moreover, one may assume that $h_0\max\vert e_j\vert\leq \min_{x_0\in F}\alpha(x_0)/N$.

Let $\epsilon>0$, $q\in \{1,\dots, p\}$ and $\theta_q=\pm 1$ be
given by Corollary \ref{cor14}. Then for $\vert x-x_0\vert<(2-\frac 1N)\alpha(x_0)$, one has
\begin{equation}\begin{aligned}
K_h^{N(x_0)+1}f(x)&\geq\dfrac1{p}K_{h,\beta_q}K_h^{N(x_0)}f(x)\geq\dfrac1{p}\int_0^1K_h^{N(x_0)}f(x+ht\theta_qe_{\beta_q})\,dt\\
&\geq c'_1\frac{h^{-d}}p\int_0^{\min(\epsilon,\tfrac{c_2}{2\max_j\vert e_j\vert})}\int_{y\in\Omega,\,|y-x-ht\theta_qe_{\beta_q}|<c_2 h}f(y)\,dydt\\
&\geq c_0'h^{-d}\int_{y\in\Omega,\,|y-x|<c_2h/2}f(y)\,dy
\end{aligned}\label{212}
\end{equation}
since for any $t\in[0, \min(\epsilon,\tfrac{c_2}{2\max_j\vert e_j\vert})]$,
$\{|y-x|<c_2h/2\}\subset\{|y-x-ht\theta_qe_{\beta_q}|<c_2h\}$.
Iterating this computation $n\leq N$ times gives \eqref{210}.

It remains to prove \eqref{211}. If $x_0\in\Omega$, the proof is
obvious. Indeed, since $\E$ spans $\R^d$, it is easy to see that for
any $\delta>0$, there exists $c_3,c_4>0$ such that for any non-negative
function $f$,
\begin{equation}
\dist(y,\partial\Omega)\geq\delta h\Longrightarrow K_h^d(f)(y)\geq c_3h^{-d}\int_{z\in\Omega,\,|y-z|<c_4h}f(z)\,dz\qquad\forall y\in\Omega.
\label{213}
\end{equation}
Suppose that $x_0\in\partial\Omega$ and denote $k=\c(x_0)$. Let
$(i_j)_{1\leq j\leq k}$, $(\beta_j)_{1\leq j\leq k}$,
$(\theta_j)_{1\leq j\leq k}$ be as in Lemma \ref{lem13}. Let
$1=\gamma_1>\gamma_2>\dots>\gamma_k>0$ and $\delta_1,\dots,\delta_k>0$
be such that for all $j$, $\gamma_j-\delta_j>\gamma_{j+1}$. Let
$G_j=[\gamma_j-\delta_j,\gamma_j]$ and $G=\Pi_{j=1}^kG_j$. In the
following computation, $c$ denotes a positive constant independant of
$f$ and $h$ that may change from line to line. Since $f$ is
non-negative,
\begin{equation}
K_h^k(f)(x)\geq p^{-k}K_{\beta_1,h}\dots K_{\beta_k,h}f(x)\geq c\int_{t\in A_h(x)}f(x+h\sum_{j=1}^k\theta_jt_je_{\beta_j})\,dt
\label{214}
\end{equation}
where $A_h(x)=\{t=(t_1,\dots,t_k)\in G,\,\forall
l=1,\dots,k,\,x+h\sum_{j=1}^l\theta_jt_je_{\beta_j}\in\Omega\}$.

Since $\theta_1e_{\beta_1}$ is strictly incoming to $H_{i_1}$,
there exists some constant $c_5,c_6>0$ such that for any $t\in I$,
\begin{equation}\begin{split}
\dist\left(x+h\sum_{j=1}^k\theta_jt_je_{\beta_j},H_{i_1}\right)&\geq c_5ht_1-c_6h(t_2+\dots+t_k)\\
&\geq c_5h(\gamma_1-\delta_1)-c_6h(\gamma_2+\dots+\gamma_k)\\
&\geq c_5h(\gamma_1-\delta_1)/2
\end{split}\label{215}
\end{equation}
by taking $\gamma_2,\dots,\gamma_k$ small with respect to $\gamma_1$.
Similarly, by taking $\gamma_j$ very small with respect to
$\gamma_{j+1}$ for $j=2,\dots,k$, there is $c_7>0$ such that for
any $j=1,\dots,k$,
\begin{equation}
\forall (t_1,\dots, t_j)\in G_1\times\dots\times G_j,\;
\dist\left(x+h\sum_{i=1}^j\theta_it_ie_{\beta_i},\R^d\setminus\overline\Omega\right)\geq c_7h.
\label{216}
\end{equation}
Hence, 
\begin{align}
K_h^kf(x)&\geq c\int_{t\in G} f\left(x+h\sum_{j=1}^k\theta_jt_je_{\beta_j}\right)\,dt\label{217}\\
\intertext{and for any $N\geq 0$}
K_h^{k+N}f(x)&\geq c\int_{t\in G} K_h^N(f)\left(x+h\sum_{j=1}^k\theta_jt_je_{\beta_j}\right)\,dt.\label{218}
\end{align}
Combining \eqref{213}, \eqref{216} and \eqref{218}, there is $c_8>0$ small enough
such that any $y\in\R^d$ such that $|
x+h\sum_{j=1}^kt_je_{\beta_j}-y|<c_8h$ belongs to $\Omega$ and hence
\begin{equation}
K_h^{d+k}f(x)\geq ch^{-d}\int_{t\in G}\int_{| x+h\sum_{j=1}^kt_je_{\beta_j}-y|<c_8h}f(y)\,dydt.
\label{220}
\end{equation}
Since, $K_{h}^kf(y)\geq p^{-k}K_{h,\beta_k}\dots K_{h,\beta_1}f(y)$, then
\begin{equation}
K_h^{d+2k}f(x)\geq ch^{-d}\int_{(t,s,y)\in B_h(x)}f(y-h\sum_{j=1}^ks_je_{\beta_j})\,dtdsdy
\label{221}
\end{equation} 
where 
\begin{multline}
B_h(x)=\left\{(t,s,y)\in G\times G\times\R^d,\;|x+h\sum_{j=1}^k\theta_jt_je_{\beta_j}-y|<c_8h\text{ and}\right.\\
\left.\forall l=1,\dots,k,\,y-h\sum_{j=l}^k\theta_js_je_{\beta_j}\in\Omega\right\}.
\label{222}
\end{multline}
Using the new variable $z=y-h\sum_{j=1}^k\theta_js_je_{\beta_j}$,
\begin{equation}
K_h^{d+2k}f(x)\geq ch^{-d}\int_{(t,s,z)\in D_h(x)}f(z)\,dtdsdz
\label{223}
\end{equation} 
with
\begin{multline}
D_h(x)=\left\{(t,s,z)\in G\times G\times\Omega,\;|x+h\sum_{j=1}^k(t_j-s_j)\theta_je_{\beta_j}-z|<c_8h\text{ and}\right.\\
\left.\forall l=1,\dots,k,\,z+h\sum_{j=1}^{l-1}\theta_js_je_{\beta_j}\in\Omega\right\}.
\label{224}
\end{multline}
Since in the above integral, $|t_j-s_j|<\delta_j$, taking the
$\delta_j$'s small enough gives
\begin{equation}
D_h(x)\supset\left\{(t,s,z)\in G\times G\times\Omega,|x-z|<c_8h/2,\; 
 \forall l=1,\dots,k,\,z+h\sum_{j=1}^{l-1}\theta_js_je_{\beta_j}\in\Omega\right\}.
\label{225}
\end{equation}
Now using \eqref{216}, it follows that 
\begin{equation}
D_h(x)\supset\left\{(t,s,z)\in G\times G\times\Omega,|x-z|<c_8h/2\right\}.
\label{226}
\end{equation}
Combined with \eqref{223}, this yields the announced result.
\end{proof}

Following the strategy of \cite{pd167}, introduce the Dirichlet form
associated to the iterated kernel $M_h^k$:
\begin{equation}
\mathcal{E}_{h,k}(u)=\left\<\left(1-M_h^k\right)u,u\right\>_{L^2(\Omega)}.
\label{227}
\end{equation}
Also, put $\Omega$ in a large box $B=]-A/2,A/2[^d$ and define an
extension map $E:L^2(\Omega)\to L^2(B)$ which is continuous from
$H^1(\Omega)$ into $H^1(B)$ and vanishes far from $\overline\Omega$.  This is possible since $\partial \Omega$
has Lipschitz regularity.  Finally, introduce the Dirichlet form on
$B$:
\begin{equation}
\tilde{\mathcal E}_h(u)=h^{-d}\int_{B\times B,|x-y|<h}|u(x)-u(y)|^2\,dxdy.
\label{228}
\end{equation}
Then Proposition \ref{prop22} easily yields the following (see \cite{pd167} for details).
\begin{lem}
  There exists $C_0,h_0>0$ such that for any $h\in]0,h_0]$ and any
  $u\in L^2(\Omega)$,
\begin{equation}
\tilde{\mathcal E}_h\left(E(u)\right)\leq C_0\left(\mathcal{E}_{h,N}(u)+h^2\|u\|_{L^2(\Omega)}^2\right).
\label{229}
\end{equation}
Moreover, any function $u\in L^2(\Omega)$ such that
\begin{equation*}
\|u\|^2_{L^2(\Omega)}+h^{-2}\left\<(1-M_{h})u,u\right\>_{L^2(\Omega)}\leq1
\end{equation*}
admits a decomposition $u=u_L+u_H$ with $u_L\in H^1(\Omega)$,
$\|u_L\|_{H^1}\leq C_1$, and $\|u_H\|_{L^2}\leq C_1 h$.
\label{lem23}
\end{lem}

We are now in position to prove the first part of Theorem \ref{thm21}.
First, assume that $M_hu=u$. Then, it follows from Proposition
\ref{prop22}, that
\begin{equation}
c_1h^{-d}\int_{\Omega\times\Omega,|x-y|<c_2h}\left(u(x)-u(y)\right)^2\,dxdy\leq\int_{\Omega\times\Omega}\left(u(x)-u(y)\right)^2M_h^N(x,dy)\,dx.
\label{230}
\end{equation}
On the other hand, the right hand side in the above inequality is
equal to $\mathcal{E}_{h,N}(u)$ which is actually equal to zero.
Hence, $u$ is constant and $1$ is a simple eigenvalue.

Using the Markov property of $M_h^N$, positivity of $\mu_h$ and the fact that $\partial\Omega$ has Lipschitz regularity, 
easily yields
\begin{equation}
\|\mu_h\|_{L^\infty\to L^\infty}=\mu_h(\Omega)\leq1-c_1h^{-d}\min_{x\in\overline\Omega}\int_{\Omega}1_{|x-y|<c_2h}\,dy<1-\delta'_0
\label{231}
\end{equation}
for some $\delta'_0>0$ independent of $h$. Working as in the proof of
Theorem 1 in \cite{pd167} shows that there exists
$\delta_0\in]0,\tfrac12[$ such that for any $u\in L^2(\Omega)$ and any
$n\geq N$,
\begin{equation}
\left\<M_h^nu,u\right\>_{L^2(\Omega)}\geq(-1+\delta_0)\|u\|_{L^2(\Omega)}^2.
\label{232}
\end{equation}
Hence, the same holds true for $n=1$ with a possibly different
$\delta_0$.

To show that there is $\delta_0>0$ sufficiently small so that the
spectrum of $M_h$ is discrete in $[1-\delta_0,1]$ it suffices to work
as in the proof of Theorem 4.6 in \cite{pd167}, using again
Proposition \ref{prop22}.

Similarly, the Weyl bound on the number of eigenvalues follows from
Lemma \ref{lem23} as in Lemma 4.8 in \cite{pd167}. This proves Part $i$.

To prove the last part of the theorem, work on the Dirichlet form is
needed. In the following, denote $\mathcal{E}_h=\mathcal{E}_{h,1}$.
Introduce the bilinear form associated with $\mathcal{E}_h$:
\begin{equation}
\mathcal{B}_h(u,v)=\<(1-M_h)u,v\>_{L^2(\Omega)},\;\forall u,v\in L^2(\Omega).
\label{233}
\end{equation}
A standard computation shows that
$\mathcal{B}_h(u,v)=\tfrac1{p}\sum_{j=1}^p\mathcal{B}_{j,h}(u,v)$ with
\begin{equation}
\mathcal{B}_{j,h}(u,v)=
\dfrac1{4h}\int_{x\in\Omega,x+te_j\in\Omega,|t|<h}\left(u(x)-u(x+te_j)\right)\left(\overline{v}(x)-\overline{v}(x+te_j)\right)\,dxdt
\label{234}
\end{equation}
\begin{lem}
  Let $\theta\in C^\infty(\overline\Omega)$ be fixed and let
  $(\varphi_h,r_h)\in H^1(\Omega)\times L^2(\Omega)$ be such that $\|
  r_h\|_{L^2(\Omega)}=O(h)$ and $\varphi_h$ converges weakly in
  $H^1(\Omega)$ to some $\varphi$. Then
\begin{equation}
\lim_{h\to0^+}h^{-2}\mathcal{B}_h(r_h,\theta)=0
\label{235}
\end{equation}
and
\begin{equation}
\lim_{h\to 0^+}h^{-2}\mathcal{B}_h(\varphi_h,\theta)=
\dfrac1{6p}\int_\Omega\left\<\nabla_\E\varphi(x),\overline{\nabla_\E\theta}(x)\right\>_{\C^p}\,dx.
\label{236}
\end{equation}
\label{lem24}
\end{lem}
\begin{proof} 
To prove \eqref{235}, observe that since $\theta$ is smooth,
\begin{equation}\begin{split}
(1-M_{j,h})\theta(x)&=\dfrac{h^{-1}}2\int_{|t|<h,x+te_j\in\Omega}\left(\theta(x)-\theta(x+te_j)\right)\,dt\\
&=\dfrac{\partial_{e_j}\theta(x)}{2h}\int_{|t|<h,x+te_j\in\Omega}t\,dt+O(h^2).
\end{split}\label{237}
\end{equation}
Denoting 
\begin{equation*}
\rho_{h}(x)=\dfrac{\partial_{e_j}\theta(x)}{2h}\int_{|t|<h,x+te_j\in\Omega}t\,dt
\end{equation*}
observe that
$\supp(\rho_h)\subset\{x\in\Omega,\,d(x,\partial\Omega)<h\}$ and
$\|\rho_h\|_{L^\infty}=O(h)$. Hence $\|\rho_h\|_{L^2}=O(h^{3/2})$ and
since $\|r_h\|_{L^2}=O(h)$, it follows that
\begin{equation}
h^{-2}\mathcal{B}_{j,h}(r_h,\theta)=h^{-2}\<r_h,(1-M_{j,h})\theta\>_{L^2}=\<h^{-1}r_h,h^{-1}\rho_h\>_{L^2}+O(h)=O(h^{1/2})
\label{238}
\end{equation}
which goes to zero as $h$ goes to zero.

To prove \eqref{236} observe that 
\begin{equation}\begin{split}
\theta(x+te_j)-\theta(x)&=t\psi(t,x)\\
\varphi_h(x+te_j)-\varphi_h(x)&=t\int_0^1\partial_{e_j}\varphi_h(x+tze_j)\,dz
\end{split}\label{239}
\end{equation}
with $\psi(t,x)$ smooth and $\psi(0,x)=\partial_{e_j}\theta(x)$. Hence
\begin{align}
h^{-2}\mathcal{B}_{j,h}(\varphi_h,\theta)
&=\dfrac1{4h^3}\int_{x\in\Omega,x+te_j\in\Omega,|t|<h,z\in[0,1]}t^2\partial_{e_j}\varphi_h(x+tze_j)\overline{\psi}(t,x)\,dtdzdx\notag\\
&=\dfrac14\int_{x\in\Omega,x+hue_j\in\Omega,|u|<1,z\in[0,1]}u^2\partial_{e_j}\varphi_h(x+huze_j)\overline{\psi}(hu,x)\,dudzdx\label{240}\\
&=\dfrac14\int_{x-huze_j\in\Omega,x+hu(1-z)e_j\in\Omega,|u|<1,z\in[0,1]}u^2\partial_{e_j}\varphi_h(x)\overline{\psi}(hu,x-huze_j)\,dudzdx.\notag
\end{align}
Taylor expansion of $\psi$ shows that
$\psi(hu,x-huze_j)=\partial_{e_j}\theta(x)+O(h)$. Hence, for any
$\delta>0$ and any $h\in]0,1]$,
\begin{align}
h^{-2}\mathcal{B}_{j,h}(\varphi_h,\theta)&=
\dfrac14\int_{x-huze_j\in\Omega,x+hu(1-z)e_j\in\Omega,|u|<1,z\in[0,1]}u^2\partial_{e_j}\varphi_h(x)\overline{\partial_{e_j}\theta}(x)\,dudzdx+O(h)\notag\\
&=I_{\delta}(h)+J_{\delta}(h)+O(h)\label{241}
\end{align}
with $I_\delta(h)$ equal to the above integral over
$d(x,\partial\Omega)\geq\delta$ and $J_\delta(h)$ the integral over
$d(x,\partial\Omega)<\delta$. Then, by Cauchy--Schwartz,
$|J_\delta(h)|\leq C(\theta)\delta^{1/2}\| \varphi_h\|_{H^1} $. On the other
hand, for any $h\in]0,\delta[$,
\begin{equation}\begin{split}
I_\delta(h)&=\dfrac16\int_{x\in\Omega,d(x,\partial\Omega)>\delta}\partial_{e_j}\varphi_h(x)\overline{\partial_{e_j}\theta}(x)\,dx\\
&=\dfrac16\int_{x\in\Omega}\partial_{e_j}\varphi_h(x)\overline{\partial_{e_j}\theta}(x)\,dx+O\left(\delta^{1/2}\|\varphi_h\|_{H^1}\right).
\end{split}\label{242}
\end{equation}
Given $\epsilon>0$, it is easy to find $\delta>0$ small enough such
that for any $h\in]0,\delta[$, $| J_\delta(h)|<\epsilon$ and $|
I_\delta(h)-\tfrac16\int_{x\in\Omega}\partial_{e_j}\varphi_h(x)\partial_{e_j}\theta(x)dx|<\epsilon$.
Now make $h\to0^+$, $\delta$ being fixed, and use the fact that $\varphi_h$
converges weakly in $H^1$ to get
\begin{equation}
\lim_{h\to 0^+}h^{-2}\mathcal{B}_{j,h}(\varphi_h,\theta)=\dfrac1{6}\int_\Omega\partial_{e_j}\varphi(x)\overline{\partial_{e_j}\theta}(x)\,dx
\label{243}
\end{equation}
and the proof is complete.
\end{proof}

To complete the proof of Theorem \ref{thm21}, denote
$|\Delta_h|=h^{-2}(1-M_h)$. Let $R>0$ be fixed and observe that if
$\nu_h\in[0,R]$ and $f_h\in L^2(\Omega)$ satisfy $| \Delta_h|
f_h=\nu_hf_h$ and $\| f_h\|_{L^2}=1$, then, thanks to Lemma
\ref{lem23}, $f_h$ can be decomposed as $f_h=\varphi_h+r_h$ with
$\|r_h\|_{L^2(\Omega)}=O(h)$ and $\varphi_h$ bounded in $H^1$. Hence
(extracting a subsequence if necessary) it may be assumed that
$\varphi_h$ weakly converges in $H^1$ to a limit $\varphi$ and that $\nu_h$
converges to a limit $\nu$. It now follows from Lemma \ref{lem24} that
for any $\theta\in C^\infty(\overline\Omega)$,
\begin{equation}
\dfrac1{6p}\int_\Omega\left\<\nabla_\E f(x),\nabla_\E\theta(x)\right\>_{\C^p}\,dx=\nu\<\varphi,\theta\>_{L^2}.
\label{244}
\end{equation}
Since $\theta$ is arbitrary, it follows that $(-\Delta_\E-\nu)\varphi=0$ and
$\partial_{n,\E}\varphi_{|\partial\Omega}=0$. In fact, this also proves that
for any $\epsilon>0$ small, there exists $h_\epsilon>0$ such that for $h\in]0,h_\epsilon]$, one has
\begin{equation}
\label{245bis}
\Spec(\vert\Delta_h\vert)\cap[0,R]\subset\cup_j[\nu_j-\epsilon,\nu_j+\epsilon]
\end{equation}
and
\begin{equation}
\sharp\Spec(|\Delta_h|)\cap[\nu_j-\epsilon,\nu_j+\epsilon]\leq m_j
\label{245}
\end{equation}
 In fact, there is equality in \eqref{245}.
The following proof is a simplification of the one in \cite{pd167}.
Proceed by induction on $j$: let $\epsilon>0$, small, be given such
that for $0\leq\nu_j\leq M+1$, the intervals
$I_j^\epsilon=[\nu_j-\epsilon,\nu_j+\epsilon]$ are disjoint. Let
$(\mu^h_j)_{j\geq 0}$ be the increasing sequence of eigenvalues of
$|\Delta_h|$, $\sigma_N=\sum_{j=1}^Nm_j$ and $(e_k)_{k\geq 0}$ an othonormal basis of
eigenfunctions of $-\Delta_\E$ such that for all
$k\in\{1+\sigma_N,\dots,\sigma_{N+1}\}$, one has
$(-\Delta_\E-\nu_{N+1})e_{k}=0$. As 0 is a simple eigenvalue of both
$-\Delta_\E$ and $|\Delta_h|$, clearly $\nu_0=\mu_0=0$ and
$m_0=1=\sharp \Spec(|\Delta_h|)\cap[\nu_0-\epsilon,\nu_0+\epsilon]$.

Suppose that for all $n\leq N$, $m_n=\sharp\Spec(|\Delta_h|)\cap
[\nu_n-\epsilon,\nu_n+\epsilon]$. Then by \eqref{245bis}, for $h\leq
h_\varepsilon$,
\begin{equation}
\mu^h_{1+\sigma_N}\geq\nu_{N+1}-\epsilon.
\label{246}
\end{equation}
By the min-max principle, if $G$ is a finite dimensional subspace of
$H^1$ with $\dim(G)=1+\sigma_{N+1}$,
\begin{equation}
\mu^h_{\sigma_{N+1}}\leq\sup_{\psi\in G,\|\psi\|=1}\<|\Delta_h|\psi,\psi\>_{L^2(\Omega)}.
\label{247}
\end{equation}
Let $G$ be the vector space spanned by the $e_{k},\ 0\leq
k\leq\sigma_{N+1}$. Then, $\dim(G)=1+\sigma_{N+1}$ and it follows from
Lemma \ref{lem24}, for any $k,k'\leq1+\sigma_{N+1}$,
\begin{equation}
\lim_{h\to0^+}h^{-2}\mathcal{B}_h(e_k,e_{k'})=\dfrac1{6p}\int_\Omega\left\<\nabla_\E e_k(x),\nabla_\E e_{k'}(x)\right\>_{\C^p}\,dx.
\label{248}
\end{equation}
Hence 
\begin{equation}
\lim_{h\to0^+}h^{-2}\mathcal{B}_h(\psi,\psi)=\dfrac1{6p}\int_\Omega|\nabla_\E \psi(x|^2\,dx\leq\nu_{N+1}
\label{249}
\end{equation}
for any $\psi\in G$ with $\Vert\psi\Vert_{L^2}=1$. Since $G$ has finite dimension, a standard
compactness argument shows that there exists $h_\epsilon>0$ such that
for any $h\in]0,h_\epsilon]$ and any $\psi\in G$ with $\|
\psi\|_{L^2}\leq 1$,
\begin{equation}
h^{-2}\mathcal{B}_h(\psi,\psi)\leq\nu_{N+1}+\epsilon.
\end{equation}
Therefore $\mu_{\sigma_{N+1}}\leq\nu_{N+1}+\epsilon$. Combining this
with \eqref{246} and \eqref{245} gives
$m_{N+1}=\sharp\Spec(|\Delta_h|)\cap[\nu_{N+1}-\epsilon,\nu_{N+1}+\epsilon]$.
The proof of Theorem \ref{thm21} is complete.

\section{Total Variation Estimates}\label{sec3}

This section gives estimates on the convergence speed of the iterated
kernel $M_h^n(x,dy)$ towards its stationary measure
$\tfrac{dy}{Vol(\Omega)}$. Recall that the total variation
$\|\mu-\nu\|_{TV}$ between two probability measures $\mu$ and $\nu$ on
$\Omega$ is defined by
\begin{equation}
\|\mu-\nu\|_{TV}=\sup|\mu(A)-\nu(A)|
\label{31}
\end{equation}
where the sup is taken over all measurable sets $A$. Equivalently,
\begin{equation}
\|\mu-\nu\|_{TV}=\dfrac12\sup_{f\in L^\infty,\|f\|_{L^\infty}=1}|\mu(f)-\nu(f)|.
\label{32}
\end{equation}
\begin{thm}
  Assume that $\E$ is weakly incoming. Then there exists $C>0$ and
  $h_0>0$ such that for all $h\in]0,h_0]$ and all $n\in\N$, the
  following estimate holds true, with $g(h)$ the spectral gap studied
  in \ref{sec2}:
\begin{equation}
\sup_{x\in\Omega}\left\|M_h^n(x,dy)-\dfrac{dy}{\vol(\Omega)}\right\|_{TV}\leq Ce^{-ng(h)}.
\label{33}
\end{equation}
\label{thm31}\end{thm}
\begin{proof}
  The proof is very close to the proof of Theorem 4.6 in \cite{pd167}
  and is just sketched for the reader's convenience. Observe first
  that $n\geq h^{-2}$ can be assumed, since otherwise the estimate is
  trivial thanks to the lower bound on the spectral gap.

  Let $\Pi_0$ be the othogonal projector in $L^2(\Omega)$ on the
  constant functions. Observe that
\begin{equation}
2\sup_{x\in\Omega}\left\|M_h^n(x,dy)-\dfrac{dy}{\vol(\Omega)}\right\|_{TV}=\left\|M_h^n-\Pi_0\right\|_{L^\infty\to L^\infty}.
\label{34}
\end{equation}
Using the spectral decomposition of $M_h$, let
$0<\lambda_{1,h}\leq\dots\leq\lambda_{j,h}\leq\dots\leq
h^{-2}\delta_0$ be such that the eigenvalues of $M_h$ in the interval
$[1-\delta_0,1]$ are the $1-h^2\lambda_{j,h}$ with associated
orthonormalized eigenfunctions
$M_h(e_{j,h})=(1-h^2\lambda_{j,h})e_{j,h}$.

Then write $M_h-\Pi_0=M_{h,1}+M_{h,2}+M_{h,3}$, so that the operators
$M_{h,1}$, $M_{h,2}$ have kernels
\begin{align}
M_{h,1}(x,y)&=\sum_{\lambda_{1,h}\leq\lambda_{j,h}\leq h^{-\alpha}}(1-h^2\lambda_{j,h})e_{j,h}(x)e_{j,h}(y)\label{35}\\
M_{h,2}(x,y)&=\sum_{h^{-\alpha}\leq\lambda_{j,h}\leq h^{-2}\delta_0}(1-h^2\lambda_{j,h})e_{j,h}(x)e_{j,h}(y)\label{36}
\end{align}
where $\alpha\in]0,2]$ is a small constant that will be chosen later.
Then
\begin{equation}
2\sup_{x\in\Omega}\left\|M_h^n(x,dy)-\dfrac{dy}{\vol(\Omega)}\right\|_{TV}\leq\sum_{j=1}^3\left\|M_{h,j}^n\right\|_{L^\infty\to L^\infty}
\label{37}
\end{equation}
and terms on the right hand side must be estimated.

From \eqref{231}, it is easy to prove that any eigenfunction
$M_h(u)=\lambda u$ with $\lambda\in]1-\delta_0,1]$ satisfies
\begin{equation}
\|u\|_{L^\infty}\leq Ch^{-d/2}\|u\|_{L^2}.
\label{38}
\end{equation}
As in \cite{pd167}, using in particular the bound on the number of eigenvalues, we show 
that for $n\in\N$,
\begin{equation}
\|M^n_{h,2}\|_{L^\infty\to L^\infty}+\|M^n_{h,3}\|_{L^\infty\to L^\infty}\leq C\big{(}(1-h^{2-\alpha})^n+(1-\delta_0)^n\big{)}h^{-3d/2}
\label{39}
\end{equation}
For $n\geq h^{-2}$, this implies that
\begin{equation}
\|M_{h,2}^n\|_{L^\infty\to L^\infty}+\|M_{h,3}^n\|_{L^\infty\to L^\infty}\leq C_\alpha e^{-nh^{2-\alpha}}.
\label{310}
\end{equation}
It remains to estimate $M_{h,1}^n$. Let $E_\alpha$ denote the space
spanned by the eigenvectors $e_{j,h}$ such that $\lambda_{j,h}\leq
h^{-\alpha}$. Then, thanks to Part $ii$ of Theorem \ref{thm21},
$\dim(E_\alpha)\leq h^{-d\alpha/2}$. As in \cite{pd167}, Lemma \ref{lem23} shows that
there exists $\alpha>0$ and $p>2$ such that for any $u\in E_\alpha$,
\begin{equation}
\|u\|_{L^p}^2\leq Ch^{-2}\left(\mathcal{E}_{h,N}(u)+h^2\|u\|_{L^2}^2\right).
\label{311}
\end{equation}
This gives the following Nash estimate, with $\tfrac1{D}=2-\tfrac4{p}>0$:
\begin{equation}
\|u\|_{L^2}^{2+\tfrac1{D}}\leq Ch^{-2}\left(\mathcal{E}_{h,N}(u)+h^2\|u\|^2_{L^2}\right)\|u\|_{L^1}^{\tfrac1{D}}\qquad\forall u\in E_\alpha.
\label{312}
\end{equation}
This inequality allows an estimate of $M_{h,1}$ from $L^1$ into $L^2$
and this leads to $\|M^{kN}_{h,1}\|_{L^\infty\to L^\infty}\leq Ce^{-k
  Ng(h)}$ for $k\geq h^{-2}$.  As $M_h$ is bounded by $1$ on
$L^\infty$ it follows that $kN$ can be replaced by $n\geq h^{-2}$ in
this estimate, and the proof of Theorem \ref{thm31} is complete.
\end{proof}

\section{Convergence of the Resolvants}\label{sec4}

Let us denote $|\Delta_h|=h^{-2}(1-M_h)$. Recall $\Delta_\E$ from
\eqref{21}. This section proves strong resolvent convergence of
$|\Delta_h|$ to $\Delta_\E$. For background and consequences, see
\cite{reedIV}.
\begin{thm}
  Let $z\in\C\setminus[0,+\infty[$ and $g\in L^2(\Omega)$. Then
\begin{equation}
\lim_{h\to 0^+}\left\|(|\Delta_h|-z)^{-1}g-(-\Delta_\E-z)^{-1}g\right\|_{L^2(\Omega)}=0.
\label{41}
\end{equation}
\label{thm41}
\end{thm}
\begin{proof}
  Let $z\in\C\setminus[0,+\infty[$ and $g\in L^2(\Omega)$ be fixed.
  For any $h>0$ let $f_h\in L^2(\Omega)$ be the solution of
  $(|\Delta_h|-z)f_h=g$. Hence
\begin{equation}
-z\<f_h,f_h\>_{L^2}+\left\<\dfrac{1-M_h}{h^2}f_h,f_h\right\>_{L^2}=\<g,f_h\>_{L^2}.
\label{42}
\end{equation}
Since $z\notin[0,\infty[$ and $\vert \Delta_h\vert$ is a positive operator, it follows that
$\|f_h\|_{L^2}\leq dist(z,[0,\infty[)^{-1}\|g\|_{L^2}$ is bounded uniformly with
respect to $h$. It follows from the above equation that there exists $C_0>0$ such that
\begin{equation}
\|f_h\|_{L^2}^2+h^{-2}\mathcal{E}_h(f_h)\leq C_0\|g\|_{L^2}^2.
\label{43}
\end{equation}
It now follows from Lemma \ref{lem24} that there exists $C>0$ depending on $z$ and $\Vert g\Vert_{L^2}$ such that
for any $h\in ]0,1]$, we can write $f_h=\varphi_h+r_h$ with $\|
\varphi_h\|_{H^1}\leq C$ and $\|r_h\|_{L^2}\leq Ch$. Let $f\in
H^1(\Omega)$ and $(h_k)_{k\in\N}$ be a sequence of positive numbers
such that $(\varphi_{h_k})_k$ converges weakly to $f$ in $H^1$. Let
$\theta\in C^\infty(\overline\Omega)$ be fixed. Then
\begin{equation}
-z\<f_{h_k},\theta\>_{L^2}+h_k^{-2}\mathcal{B}_{h_k}(f_{h_k},\theta)=\<g,\theta\>_{L^2}
\label{44}
\end{equation}
and taking the limit $k\to\infty$ it follows from Lemma \ref{lem24}
that
\begin{equation}
-z\<f,\theta\>_{L^2}+\dfrac1{6p}\int_\Omega\nabla_\E f(x)\overline{\nabla_\E\theta(x)}\,dx=\int_{\Omega}g(x)\overline{\theta(x)}\,dx.
\label{45}
\end{equation}
Since $\theta$ is arbitrary, this implies $(-\Delta_\E-z)f=g$ and
$\partial_{n,\E}f_{|\partial\Omega}=0$.
Since, this is true for any subsequence $(h_k)$, this shows that $\Vert f_h-f\Vert_{L^2}\rightarrow 0$ when $h\rightarrow 0$, which is exactly \eqref{41}.
\end{proof}

\section{Some Generalizations}\label{sec5}

Here we present a possible generalization of the previous results. It
is still assumed that $\Omega$ is a convex polytope in $\R^d$. Suppose
that $E\subset\R^d$ is endowed with a Borel probability measure $\mu$. For
any $e\in E$, define
\begin{equation}
K_{e,h}f(x)=\dfrac12\int_{t\in [-1,1],x+hte\in\Omega}f(x+hte)\,dt
\label{51}
\end{equation}
and 
\begin{equation}
K_hf(x)=\int_{e\in E}K_{e,h}f(x)\,d\mu(e).
\label{52}
\end{equation}
The associated Metropolis operator is defined by
$M_hf(x)=m_h(x)f(x)+K_hf(x)$ with $m_h(x)=1-K_h(1)$.
\begin{defin}
  Say that $(E,\mu)$ is weakly incoming to $\Omega$ if for any
  $x_0\in\partial\Omega$ there exists $\epsilon>0$,
  $\theta\in\{\pm1\}$ and a measurable subset $F\subset E$ such that
  $\mu(F)>0$ and
\begin{equation}
\c(x_0+\theta te)<\c(x_0)\qquad\forall t\in ]0,\epsilon],\ \forall e\in F.
\label{53}
\end{equation}
\label{def51}
\end{defin}

\begin{lem}\label{lem51} There exists some measurable subsets $F_1,\ldots, F_d\subset E$ such that $\mu(F_j)>0$ for all $j$ and any 
$(f_1,\ldots,f_d)\in\Pi_{j=1}^dF_j$ spans $\R^d$.
Moreover the sets $F_j$ can be chosen with arbitrary small diameters.
\end{lem}
\begin{proof} From the same argument as in remark following \ref{def11},  we can easily see that $\mu$ can not be supported in an hyperplane of $\R^d$.
Let us prove by induction that for $k=1,\ldots, d$, there exists $F_1,\ldots F_k\subset E$ such that $\mu(F_j)>0$ for all $j$ and for any 
$(f_1,\ldots,f_k)\in\Pi_{j=1}^kF_j$, $rank(f_1,\ldots, f_k)=k$.

If $k=1$, it suffices to take $F_1\subset F\setminus\{0\}$ with $\mu(F_1)>0$, which is possible thanks to the fact that $F$ is weakly incoming to $\Omega$.

Assume that the property holds true at rank $k-1<d$. There exists  $F_1,\ldots F_{k-1}\subset E$ such that $\mu(F_j)>0$ for all $j$ and any 
$(f_1,\ldots,f_{k-1})\in\Pi_{j=1}^{k-1}F_j$, $H=span(f_1,\ldots, f_{k-1})$ has dimension $k-1$. Since $\supp(\mu)$ is not contained in $H$, there exists $F_k\subset F\setminus H$ with $\mu(F_k)>0$. Then $F_1,\ldots,F_k$ satisfy the property at rank $k$.

The fact that we can take $diam(F_j)$ arbitrary small can be shown as follows. Let $\epsilon>0$ and assume by contradiction that there exists $j_0$ 
such that for any $f\in F_{j_0}$, $\mu(B(f,\epsilon)\cap F_{j_0})=0$. Then any compact subset of $F_{j_0}$ would have measure zero, which is impossible since $\mu(F_{j_0})>0$.
\end{proof}

Introduce the following differential operators associated to the set $E$:
\begin{equation}
\nabla_E:\,H^1(\Omega)\to L^\infty\left(E,L^2(\Omega)\right)
\label{54}
\end{equation}
defined by $\nabla_Eu(e,x)=\<\nabla u(x),e\>_{\C^d}$ for any $(e,x)\in E\times\Omega$;
\begin{equation}
\div_E:\,L^\infty\left(E,H^1(\Omega)\right)\to L^2(\Omega)
\label{55}
\end{equation}
defined by $\div_Ef(x)=\int_E\<\nabla_xf(e,x),e\>_{\C^d}d\mu(e)$ for any $x\in\Omega$; and
\begin{equation}
\Delta_E:\,H^2(\Omega)\to L^2(\Omega)
\label{56}
\end{equation}
given by $\Delta_E=\tfrac16\div_E\nabla_E$.

Define also the following trace operator:
\begin{equation}
\gamma_E^0:\left\{f\in L^\infty\left(E,H^1(\Omega)\right),\,\div_E f\in L^2(\Omega)\right\}\to H^{-\tfrac12}(\partial\Omega)
\label{57}
\end{equation}
by
\begin{equation}
\int_{\partial\Omega}\gamma_E^0f(x)v_{|\partial\Omega}(x)\,d\sigma(x)=\int_\Omega\div_Ef(x)v(x)\,dx+
\int_E\int_\Omega f(e,x)\nabla_Ev(e,x)\,dxd\mu(e)
\label{58}
\end{equation}
for any $v\in H^1(\Omega)$. Observe that if $f\in
L^\infty(E,C^1(\overline\Omega))$, then
\begin{equation}
\gamma_E^0f(x)=\int_E\left\<e,n(x)\right\>_{\C^d}f(x,e)\,d\mu(e)
\label{59}
\end{equation}
where $n(x)$ denotes the unit outgoing normal vector to the boundary
$\partial\Omega$ at point $x$.

For $u\in H^1(\Omega)$ such that $\Delta_Eu\in L^2(\Omega)$, the
function $f=\nabla_Eu$ satisfies $\div_E f\in L^2(\Omega)$. Hence, the
operator
\begin{equation}
\gamma_E^1:\left\{u\in H^1(\Omega),\,\Delta_Eu\in L^2(\Omega)\right\}\to H^{-\tfrac12}(\partial\Omega)
\label{510}
\end{equation}
defined by $\gamma_E^1u(x)=\gamma_E^0\nabla_Eu(x)$ is continuous.

Finally, introduce the following quadratic form on $H^1(\Omega)$:
\begin{equation}
\mathcal{E}_E(u)=\dfrac16\int_E\int_{\Omega}|\nabla_Eu(e,x)|^2\,dxd\mu(e)\qquad\forall u\in H^1(\Omega).
\label{511}
\end{equation}

From  Lemma \ref{lem51}  it follows that, since $E$ is
weakly incoming to $\Omega$, there exists some subsets $F_1,\dots,F_d$
with arbitrary small diameters and $\mu(F_j)>0$ such that any $(f_1,\dots,f_d)\in
F_1\times\dots\times F_d$ spans $\R^d$. Taking the diameter of the
$F_j$ sufficiently small, it is easy to show that there exists $C>0$
such that for any $u\in H^1(\Omega)$,
\begin{equation}
\dfrac1{C}\|\nabla u\|_{L^2(\Omega)}^2\leq\mathcal{E}_E(u)\leq C\|\nabla u\|_{L^2(\Omega)}^2.
\label{512}
\end{equation}

Then, the operator $-\Delta_E=-\tfrac16\div_E\nabla_E$ with domain
$D(-\Delta_E)=\{u\in H^1(\Omega),\; \Delta_E u\in
L^2(\Omega),\,\gamma_E^1u=0\}$ is the self-adjoint realization of the
Dirichlet form $\mathcal{E}_E$.  Moreover, it follows from \eqref{512}
that $-\Delta_E$ has compact resolvant. Denote its spectrum by
$\nu_0=0<\nu_1<\nu_2<\dots$ and by $m_j$ the multiplicity associated
to $\nu_j$. Observe that $m_0=1$.
\begin{thm}
  Suppose that $(E,\mu)$ is weakly incoming to $\Omega$, then the
  following hold true.
\begin{enumerate}
\item[i)] There exists $h_0>0,\ \delta_0\in]0,\tfrac12[$ and a
  positive constant $C$ such that for any $h\in]0,h_0]$, the spectrum
  of $M_h$ is a subset of $[-1+\delta_0,1]$, 1 is a simple eigenvalue
  and $\Spec(M_h)\cap[1-\delta_0,1]$ is discrete.

\item[ii)] For any $h\in]0,h_0]$ and
  $0\leq\lambda\leq\delta_0h^{-2}$, the number of eigenvalues of
  $M_{h}$ in $[1-h^2\lambda,1]$ (with multiplicity) is bounded by
  $C(1+\lambda)^{d/2}$.
  
\item[iii)] For any $R>0$ and $\varepsilon>0$ such that
  $\nu_{j+1}-\nu_j>2\varepsilon$ for $\nu_{j+2}<R$, there exists
  $h_1>0$ such that one has for all $h\in]0,h_1]$,
\begin{equation}
\Spec\left(\frac{1-M_{h}}{h^2}\right)\cap]0,R]\subset\cup_{j\geq 1}[\nu_j-\varepsilon,\nu_j+\varepsilon]
\label{513}
\end{equation}
and the number of eigenvalues of $\frac{1-M_{h}}{h^2}$ in the interval
$[\nu_j-\varepsilon,\nu_j+\varepsilon]$ is equal to $m_j$.
\end{enumerate}
\label{thm52}
\end{thm}

Here are two examples of $(E,\mu)$ which are weakly incoming to
$\Omega$. The first is the case where $E=\{e_1,\dots,e_p\}$ is
discrete and $\mu$ is simply the measure
$\tfrac1{p}\sum_{j=1}^p\delta_{e=e_j}$. Then it suffices to assume
that $E$ is weakly incoming to $\Omega$ in the sense of \ref{def11}.
Moreover, in that case the conclusion of Theorem \ref{thm52} are
exactly those of Theorem \ref{thm21}.

A second example is the following. Let $E$ be equal to the sphere
$S^{d-1}$ and $\mu=d\sigma_d$ be the surface measure. Assume that
$\rho:S^{d-1}\to\R^+$ is a continuous function such that
$\int_{S^{d-1}}\rho(\omega)d\sigma_d(\omega)=1$ and let
$\mu=\rho(\omega)d\sigma_d(\omega)$. Then $(E,\mu)$ will be weakly
incoming to $\Omega$ iff there exists a family of vectors
$e_1,\dots,e_p\in\supp(\rho)$ such that $(e_1,\dots,e_p)$ is weakly
incoming in the sense of  \ref{def11}. For instance, if
$\rho$ is strictly positive on $S^{d-1}$ then these assumptions are
automatically satisfied.

The proof of Theorem \ref{thm52} is very close to that of Theorem
\ref{thm21} and only the main steps are given.  The following
proposition is a version of Lemma \ref{lem13} adapted  to the present setting.
\begin{prop}
  Assume that $(E,\mu)$ is weakly incoming to $\Omega$, let $x_0\in\overline\Omega$ and denote $k=\c(x_0)$. There exists
  $\epsilon>0$ and some subsets $F_1,\dots,F_k\subset E$ such that
  $\mu(F_i)>0$ for all $i=1,\dots,k$ and
\begin{itemize}
\item there exists $r_0>0$ and
  $I\subset\{1,\dots,m\}$ with $\sharp I=k$ such that
\begin{equation}
\Omega\cap B(x_0,r_0)=\left(\cap_{i\in I}^kH_{i}^+\right)\cap B(x_0,r_0);
\label{514}
\end{equation}
\item there exists
  $\theta_1,\dots,\theta_k\in\{\pm1\}$  and a bijection $\{1,\ldots,k\}\ni n\mapsto i_n\in I$ such that for any $n=1,\dots,k$ and any
  $f_n\in F_{n}$,
\begin{equation}
\theta_nf_n\text{ is strictly incoming to }H_{i_n}
\label{515}
\end{equation}
and
\begin{equation}
\theta_nf_n\text{ is incoming to }H_{i_m}\qquad\forall m>n.
\label{516}
\end{equation}
\end{itemize}
Moreover the sets $F_1,\dots,F_k$ can be chosen with arbitrary small
diameter.
\label{prop53}
\end{prop}
\begin{proof}
  First, it is clear that \eqref{514} holds true.  We prove \eqref{515} and \eqref{516} by induction on $k=\c(x_0)$. For $k=0$ there is nothing to prove.

Assume now that the property holds true for all $x_0'$ 
such that $\c(x_0')\leq k-1$ and suppose that $\c(x_0)=k$. Since
$(E,\mu)$ is weakly incoming to $\Omega$, there exists $F\subset E$,
$\theta_1\in\{\pm1\}$ and $\epsilon>0$ such that $\c(x_0+t\theta_1
f)<\c(x_0)$ for all $t\in]0,\epsilon]$. Assume without loss of
generality that $\theta_1=1$.  Since $\mu(F)>0$, there exists $f^0\in
F$ such that for all $\rho>0$, $\mu(B(f^0,\rho)\cap F)>0$ and
\begin{equation}
\c(x_0+tf)<\c(x_0)\qquad\forall f\in B(f^0,\rho)\cap F,\ \forall t\in]0,\epsilon].
\label{519}
\end{equation}
In particular, there exists $q_1\in\{1,\dots,k\}$ and $i_1,\ldots ,i_{q_1}\in I$  such that
\begin{equation}
f^0\text{ is strictly incoming to }H_{i_q}\qquad\forall q=1,\dots,q_1
\label{520}
\end{equation}
and 
\begin{equation}
f^0\text{ is parallel to }H_{i_q}\qquad\forall q\geq q_1+1.
\label{521}
\end{equation}
Let $F_q=B(f_0,\rho)\cap F$ with $\rho>0$ for $q=1,\ldots ,q_1$. Then $\mu(F_q)>0$ and it follows from \eqref{520}
that for $\rho$ small enough, any $f\in F_q$ is strictly incoming to $H_{i_q}$. Moreover, thanks to \eqref{519}, any $f\in F_q$ is incoming to $H_{i}$ for $i\in I\setminus\{i_1,\ldots, i_{q_1}\}$.
Then we can use the induction hypothesis with  $x'_0=x_0+\epsilon f_0$ close to $x_0$ such that $\c(x'_0)=k-q_1<k$ to build $F_{q_1+1},\ldots, F_k$.
The statement concerning the diameter of the $F_j$ is a trivial
consequence of the construction.
\end{proof}
\begin{cor}
  Assume that $(E,\mu)$ is weakly incoming to $\Omega$ and let
  $x_0\in\overline\Omega$. Then there exists $r_0>0$, $\epsilon>0$,
  $F\subset E$ with $\mu(F)>0$ and $\theta\in\{\pm1\}$ such that
\begin{equation}
\forall f\in F,\ \forall x\in
B(x_0,r_0)\cap\Omega,\ \forall t\in[0,\epsilon],\;x+t\theta f\in\Omega
\label{526}
\end{equation}
\label{cor54}
\end{cor}

Using these results and working as in \ref{sec2} easily proves the
following.
\begin{prop}
There exists $N\in\N$ and $c_1,c_2>0$ such that
\begin{equation}
M_h^N(x,dy)=\mu_h(x,dy)+c_1h^{-d}1_{| x-y|<c_2h}\,dy
\label{527}
\end{equation}
where for all $x\in\Omega$, $\mu_h(x,dy)$ is a positive Borel measure.
\label{prop55}
\end{prop}
\begin{proof}
  The starting point of the proof is to observe that for any $k\in\N$
  and any non-negative function $f$,
\begin{equation}
K_h^kf(x)\geq\int_{e_1\in F_{1}}\dots\int_{e_k\in F_{k}}K_{h,e_1}\dots K_{h,e_k}f(x)\,d\mu(e_k)\dots d\mu(e_1)
\label{528}
\end{equation}
for any $F_1,\dots,F_k\subset E$.  Then the proof is the same as the
proof of Proposition \ref{prop22}.
In fact, \eqref{213} remains valid thanks to Lemma \ref{lem51}. Then we can mimick the end of the proof, using the fact that in Proposition
\ref{prop53} the set $F_j$ can be chosen with arbitrary small diameter. Details are left to the reader.
\end{proof}

Proposition \ref{prop55} implies a lemma analogous to Lemma
\ref{lem23} for the operator $M_h$ considered in this section. In
particular, any function $u\in L^2(\Omega)$ satisfying
\begin{equation}
\|u\|_{L^2}^2+\left\<(1-M_h)u,u\right\>_{L^2}\leq1
\label{529}
\end{equation}
admits a decomposition $u=u_L+u_H$ with $\|u_L\|_{H^1}\leq1$ and
$\|u_H\|_{L^2}=O(h)$.  Using Proposition \ref{prop53} and the
generalization of Lemma \ref{lem23} gives Parts $i$ and $ii$ of
Theorem \ref{thm52}.

Part $iii$ is implied by the following lemma (where $\mathcal{B}_h$
still denotes the Dirichet form associated to $M_h$).
\begin{lem}
  Let $\theta\in C^\infty(\overline\Omega)$ be fixed and let
  $(\varphi_h,r_h)\in H^1(\Omega)\times L^2(\Omega)$ be such that
  $\|r_h\|_{L^2(\Omega)}=O(h)$ and $\varphi_h$ converges weakly in
  $H^1(\Omega)$ to some $\varphi$. Then
\begin{equation}
\lim_{h\to0^+}h^{-2}\mathcal{B}_h(r_h,\theta)=0
\label{530}
\end{equation}
and
\begin{equation}
\lim_{h\to0^+}h^{-2}\mathcal{B}_h(\varphi_h,\theta)=\dfrac16\int_E\int_\Omega\nabla_E\varphi(e,x)\overline{\nabla_E\theta}(e,x)\,dxd\mu(e).
\label{531}
\end{equation}
\label{lem56}
\end{lem}
\begin{proof}
The proof is the same as that of Lemma \ref{lem24}.
\end{proof}

Total variation estimates for rates of convergence now follow as in \ref{sec3}.

\bibliographystyle{amsplain}
\bibliography{PDall,PDnone}

\end{document}